\documentclass{amsart}
\usepackage{amsmath,amsfonts,amssymb,graphicx}

\newtheorem{thm}{Theorem}[section]
\newtheorem{cor}[thm]{Corollary}
\newtheorem{lem}[thm]{Lemma}

\theoremstyle{definition}
\newtheorem{defn}[thm]{Definition}
\theoremstyle{remark}

\newtheorem{ex}[thm]{Example}
\numberwithin{equation}{section}

\begin{document}

\title[Two bounds on the  noncommuting graph]
{Two bounds on the noncommuting graph}

\author[S. Nardulli]{Stefano Nardulli}
\address{\newline
Instituto de Matem\'atica, Universidade Federal do Rio de Janeiro,  Av. Athos da Silveira Ramos 149, Centro de Tecnologia,  Bloco C, Cidade Universit\'aria, Ilha do Fund\~ao,  Caixa Postal 68530, 21941-909,  Rio de Janeiro, Brasil.}
\email{nardulli@im.ufrj.br}

\author[F.G. Russo]{Francesco G. Russo}
\address{
%\newline
%Instituto de Matem\'atica, Universidade Federal do Rio de Janeiro,  Av. Athos da Silveira Ramos 149, Centro de Tecnologia,  Bloco C, Cidade Universit\'aria, Ilha do Fund\~ao,  Caixa Postal 68530, 21941-909,  Rio de Janeiro, Brasil
%\newline
%and
\newline
Department of Mathematics and Applied Mathematics, University of Cape Town, Private Bag X1, Rondebosch 7701,
Cape Town, South Africa.}
\email{francescog.russo@yahoo.com}

\date{\today}

%% Four or five keywords or phrases

\begin{abstract}Erd\H{o}s introduced the  noncommuting graph, in order to  study the number of commuting elements in a finite group. Despite the use of combinatorial ideas, his methods involved several techniques of classical analysis.
The interest for this graph is becoming relevant in the last years for various reasons. Here we deal with a numerical aspect, showing for the first time an isoperimetric inequality  and an analytic  condition in terms of  Sobolev inequalities. This  last result holds in the more general context of weighted locally finite  graphs.
\end{abstract}

\maketitle

\section{Terminology and preliminary notions}

If  $\Gamma$  denotes a locally finite  graph (i.e.: each vertex of $\Gamma$ has a finite number of neighbors) with  vertex set $V$ and  edge set $E$, two elements  $x, y \in V$ are in the relation $x \sim y$ if $x$ and $y$ are adjacent and joined by an edge $xy$. For a subset $\Omega \subseteq V$,  \[\partial \Omega= \{xy \ | \ x \in \Omega \ \mathrm{and} \ y \in V-\Omega   \},\] is the set of edges which join a vertex in $\Omega$ with a vertex outside $\Omega$.  In presence of an orientation,  each edge in $\partial \Omega$ is oriented so that it points outwards from $\Omega$.  To  $\Gamma$, we associate  \textit{the edge weight} $\sigma_{xy}>0$ for  each  $xy \in E$, so for any $S \subseteq E$ we define the measure \[\sigma(S)=\sum_{xy\in S}\sigma_{xy}.\] Extending the function $\sigma_{xy}$ by zero to those $x, y$ which are not neighbors, we get a symmetric function from $V \times V$ to $]0,+\infty[$. It will be also useful to introduce \textit{the vertex weight}
\[\mu_x: x \in V \longmapsto \mu_x=\sum_{y : y \sim x} \sigma_{xy} \in ]0,+\infty[.\]
In case $\sigma_{xy}=1$ for all $xy \in E$ (for instance, in unweighted  graphs),  \[\mu_x=\mathrm{deg}(x)=|\{yx \ | \ y \sim x\}|\]  is the \textit{degree} of $x$, that is, the number of neighbors of the vertex $x$. On the other hand, it is well defined the positive measure \[\mu : \Omega \subseteq V \ \longmapsto \mu(\Omega)=\sum_{x \in \Omega} \mu_x \in ]0,+\infty[.\]
If $\Gamma$ is equipped with  $\sigma$ and $\mu$ as above,  we say that it is a  \textit{weighted graph}. In particular, if $\Gamma_G$ is the noncommuting graph of a finite group $G$ (i.e.: recall from \cite{aam} that $\Gamma_G$  is defined by vertices $x,y \in G-Z(G)=V$ joined by an edge $xy \in E$ if $x$ do not commute with $y$), there is neither weight  nor orientation, so 
$\mu(\Omega)=\sum_{x \in \Omega} \mathrm{deg}(x)$ and  $\sigma(\partial \Omega)=|\partial \Omega|$.
Important contributions on $\Gamma_G$ can be found in  \cite{aam,   dar,  mog}, but the reader may refer  to \cite{fgr} for a recent survey \footnote{This graph appears originally in certain combinatorial problems in group theory, related to conjectures of  Erd\H{o}s on the number of commuting elements in a group (see \cite{neum}). A probabilistic version of these ideas  can be found in \cite{hof1, hof2}.}.  $\Gamma_G$ has interesting properties: it is always connected, of diameter 2 and hamiltonian (see \cite[Propositions 2.1, 2,2]{aam}); moreover  the planar and the regular cases are classified by \cite[Propositions 2.3, 2.6]{aam}. To the best of our knowledge, there are no isoperimetric inequalities on its invariants and  we are going to show one of these here for the first time.

 Following \cite[\S 5.2]{chung4}, 
 it is possible to define   \textit{the gradient operator} 
\[ \nabla :  f \in \mathbb{R}^{V \times V} \longmapsto \nabla f=\nabla_{xy}f=f(y)-f(x) \in \mathbb{R},\]
where $\mathbb{R}^{V \times V}$ is the set of all functions from $V \times V$ to $\mathbb{R}$,  and $\nabla_{xy}$ denotes the fact that there is a dependence from $x,y \in V$ in the definition of $\nabla$. Consequently, 
\[\Delta :  f \in \mathbb{R}^V \longmapsto \Delta f(x)= \frac{1}{\mu_x} \sum_{y :  y \sim x} (\nabla_{xy}f) \sigma_{xy} \in \mathbb{R}\]
is  \textit{the Laplace operator }.  A natural variation of the Green's Formula is 
 \[{\underset{x \in \Omega}{\underset{|\Omega|< \infty}\sum}} \Delta f(x) \mu_x
 ={\underset{x \in \Omega, y \in V-\Omega}{\underset{|\Omega|< \infty}\sum}} (\nabla_{xy} f) \sigma_{xy}={\underset{e \in \partial \Omega}{\underset{|\Omega|< \infty}\sum}} (\nabla_e f) \sigma_e\]
and, if $f, g \in \mathbb{R}^V$ with either $f$ or $g$ of finite support, then  
 \[\sum_{x \in V} \Delta f(x) g(x) \mu_x=-\frac{1}{2} \sum_{x,y \in V} (\nabla_{xy} f)  (\nabla_{xy} g) \sigma_{xy}=-\sum_{e \in E} (\nabla_e f) (\nabla_e g)\sigma_e.\]
We will consider  distance functions on $V$, inspired by analogous contexts of riemannian geometry in \cite{nardulli2, nardulli1}. The \textit{graph distance} $\rho_\xi(x)$ between $x \in V$ and the fixed vertex $\xi \in V$ is the number of edges in a shortest path (also called a \textit{graph geodesic}) connecting them and, in particular,  \[\rho : (\xi, x) \in V \times V \longmapsto \rho(\xi,x)=\rho_\xi(x) \in \mathbb{N} \ \ \mathrm{and} \ \ \rho_\xi : x \in V \longmapsto \rho_\xi(x) \in \mathbb{N}.\]
This is also known as the geodesic distance and we note that there are more than one shortest path between two vertices\footnote{ If there is no path connecting the two vertices, i.e., if they belong to different connected components, then conventionally the distance is defined as infinite. We also note that in  case of a directed graph the distance $\rho_\xi(x)$ is defined as the length of a shortest path from $\xi$ to $x$ consisting of arcs, provided at least one such path exists.}. In contrast with the case of undirected graphs, one may have  $\rho$ is not symmetric a priori. But we only deal with graphs possessing a distance function as $\rho$ and all we have said up to now is of course true for finite graphs (in particular, for $\Gamma_G$). This allows us to define, fixed $r>0$, a ball
$B_\xi(r)=\{x \in V \ | \ \rho_\xi(x) <r\}.$ We show mainly two results in the present paper. One is a specialization to $\Gamma_G$ of theorems in \cite{chung4}. This provides an isoperimetric inequality for $\Gamma_G$, which is unknown up to now. The second main result has more general interest and shows a characterization  in terms of a Nash--type inequality of certain locally finite weighted graphs, which generalise the noncommuting graph.

\section{First result}

Following \cite{chung1, chung4}, we may restrict the investigations to  graphs, whose  geometric properties are analogous with some classical notions of the riemannian manifolds (see \cite{bakry1,  hebey}). For a wieghted graph $\Gamma$ with a distance $\rho$,  the positive quantity \[\mu^\xi_x={\underset{\rho_\xi(y)<\rho_\xi(x)}{\underset{y : y \sim x \ } \sum}} \sigma_{xy},\] 
clearly satisfies $\mu^\xi_x<\mu_x$ and allows us to introduce the ratio
\[\nu_r=\inf \left\{\frac{\mu_x}{\mu^\xi_x} \ \Big| \ \xi \in V, x \in B_\xi(r)\right\},\]
which correspond to the notion of \textit{relative isoperimetric dimension} in \cite{aubin, bakry1,  hebey}.
In  this spirit, Chung and others introduced the  so called  $P(\delta,\iota,R_0)$ \textit{property} .

\begin{defn}[See \cite{chung4}]\label{p}We say that $(\Gamma,\sigma)$ has $P(\delta, \iota, R_0)$, when :
\begin{itemize}
\item[(i)] $|\nabla_{xy} \rho_\xi| \le 1$ for any $\xi, x,y \in V$;
\item[(ii)] $\exists$ a function $q_\xi (x)$ and three  constants $\iota \ge 1$, $\delta>0$ and $R_0>0$ such that
\begin{itemize}
\item[(1)]$q_\xi(x) \ge 0$ for all $x \in V$, and $q_\xi(x)=0$ if and only if $x=\xi$;
\item[(2)]$|\nabla_{xy} q_\xi| \le \rho_\xi(x) + \iota$ for all $\xi \in V$ and $x,y \in B_\xi(R_0)$;
\item[(3)]$\Delta q_\xi(x) \ge \delta$  for all $\xi \in V$ and $x \in B_\xi(R_0)$;
\end{itemize}
\item[(iii)]$n= \delta \nu_{R_0 +1} \ge 1$.
\end{itemize}
\end{defn}

The presence of an isoperimetric inequality  can be deduced from $P(\delta, \iota, R_0)$. 

\begin{thm}[See \cite{chung4}, Theorem 6.3]\label{iso1} If a weighted graph $(\Gamma, \sigma)$ has  $P(\delta, \iota, R_0)$, then the following isoperimetric inequality is true
\[\sigma(\partial \Omega) \ge c \ \mu(\Omega)^{1-\frac{1}{n}},\]
where $\Omega \subseteq V$ is finite, $\omega= \inf \{\mu_x \ | \ x \in V\}$, $\omega'=\inf \{\sigma_{xy} \ | \ x\sim y, x,y \in V \}$ and \[ c=\frac{\omega' \omega^{\frac{1}{n-1}}}{4^{n+3} \nu_{R_0 +1} \iota e^{2n}}.\]
\end{thm}

An inequality of Sobolev type (see \cite{bakry1}) is recalled below in our context. Note that the presence of an isoperimetric inequality is requested in the assumptions.

\begin{thm}[See \cite{chung4}, Theorem 7.6]\label{sob1} If a weighted graph $(\Gamma, \sigma)$ possess a finite subset $\Omega \subseteq V$ of $\mu(\Omega)<v_0 $ such that 
$\sigma(\partial \Omega) \ge c \mu(\Omega)^{1-\frac{1}{n}}$ for some $c,v_0>0$ and $n>1$, then
\[C(n,p)  \left( {\underset{x,y \in V}{\underset{y : y \sim x} \sum}} | f(y)-f(x)|^p \sigma_{xy}\right)^{\frac{1}{p}} + c K(\Omega_0) \left(\sum_{x \in V}|f|^p \mu_x \right)^{\frac{1}{p}}\]
\[ \ge \frac{c}{2^{1+\frac{1}{n}-\frac{1}{p}}} \left(\sum_{x \in V} |f|^{\frac{np}{n-p}}\right)^{\frac{n-p}{np}}\leqno{(\flat)}\]
for any $f \in \mathbb{R}^V$ of finite support, where (with the meaning of Theorem \ref{iso1}) \[  c=\frac{\omega' \omega^{\frac{1}{n-1}}}{4^{n+3} \nu_{R_0 +1} \iota e^{2n}},\]   $C(n,p)>0$ is a  positive constant, \[\Omega_0=\{x \in V \ | \ |f(x)|>0\}\] and $$K(\Omega_0)= \begin{cases}  0  & \ \ \ if \ \ \mu(\Omega_0) \le v_0 \\
v^{-1}_0\mu(\Omega_0)^{1-\frac{1}{n}} & \  \ \ if \ \ \mu(\Omega_0) > v_0.
\end{cases}$$
\end{thm}

What we said until now can be tested for the noncommuting graph.

\begin{lem}\label{l:1}
$\Gamma_G$ satisfies (i)--(iii) of Definition \ref{p}.
\end{lem}

\begin{proof} We begin to check (i). Recall that $\mathrm{diam} \ \Gamma_G=2$ and $|\nabla_{xy} \rho_\xi|=|\rho_\xi(y) - \rho_\xi(x)|$. Of course
 $|\nabla_{xy} \rho_\xi|=0$, whenever $x=y=\xi$. Assume $x \neq y$ and $x=\xi$. Since $\mathrm{diam} \ \Gamma_G=2$, we  have 3 points in $V$ and two of them coincide, hence $\rho_\xi(y)=1$ and $\rho_\xi(x)=0$. The same argument applies when $x \neq y$ and $y=\xi$. Then in both cases $|\nabla_{xy} \rho_\xi|=1$. Assume now $x \neq y$, $x \neq \xi$ and $y \neq \xi$. Again the condition $\mathrm{diam} \ \Gamma_G=2$ implies $\rho_\xi(x)= \rho_\xi(y)=1$ and so  $|\nabla_{xy} \rho_\xi|=0$. This allows us to conclude that $|\nabla_{xy} \rho_\xi|=1$ for all $x,y,\xi \in V$. About (ii) of Definition \ref{p} it is enough to put $q_\xi(x)=\frac{1}{2}\rho^2_\xi(x).$ In fact one can check easily (ii.1). About (ii.2) \[\nabla_{xy} q_\xi = q_\xi(y)-q_\xi(x)=\frac{1}{2} (\rho^2_\xi(y)-\rho^2_\xi(x))=\frac{1}{2}\underbrace{(\rho_\xi(y)-\rho_\xi(x))}_{\le 1 \ \mathrm{from} \  \mathrm{(i)} \ \mathrm{above} } \ (\rho_\xi(y)+\rho_\xi(x))\]\[ \le \frac{1}{2} (\rho_\xi(y) + \rho_\xi(x))   \le 1 + \rho_\xi(x) .\]
Finally, for any $\xi \in V$ and $x \in B_\xi(2)=B_\xi(R_0)$ we have 
\[ 2 \ \mathrm{deg}(x) \ \Delta q_\xi(x)= \sum_{y : y \sim x}  (\rho_\xi(y)-\rho_\xi(x)) \ (\rho_\xi(y)+\rho_\xi(x))  \ge 1\] and so (ii.3) is realized with $\delta=1$.  (iii) is satisfied with $n=\nu_3$, but $\mathrm{diam} \ \Gamma_G=2$ implies $\nu_3=\nu_2$ and so $n=\nu_2$ is better.
\end{proof}

The previous lemma provides information, which we summarize below.

\begin{cor}\label{c:1}  $\Gamma_G$ has $R_0=\delta=\iota=1$ and $n=\nu_2$ in Definition \ref{p}.
\end{cor}

\begin{proof}See proof of Lemma \ref{l:1}.
\end{proof}

Now our first main result can be stated.

\begin{thm}\label{spec1}  $\Gamma_G$ satisfies the isoperimetric inequality
\[|\partial \Omega| \ge c \left( \sum_{x \in \Omega} \mathrm{deg}(x)\right)^{1-\frac{1}{\nu_2}}, \]
where  $\Omega \subseteq G-Z(G)$, $\omega=\inf \{\mathrm{deg} \ (x) \ | \ x \in G-Z(G)\}$ and
\[ \ c=\frac{\omega^{1-\frac{1}{\nu_2}}}{4^{\nu_2 +3} \ \nu_2 \ e^{2 \nu_2}}.\]
\end{thm}

\begin{proof} We specialize Thereom \ref{iso1}, by the use of Lemma \ref{l:1} and Corollary \ref{c:1}. 
\end{proof}

There are  difficulties of computation for $\nu_2$ already for groups of order 8.

\begin{ex}Let $G=Q_8=\{1,-1,i,-i,j,-j,k,-k \ | \ ij=k, jk=i, ki=j, i^2=j^2=k^2=-1\}$ be the quaternion group of order 8.
This presentation is not elegant in terms of generators and relations, but  very useful for our aims. In fact we can see immediately that $\Gamma_G$ has $|V|=|Q_8-Z(Q_8)|=|\{i,-i,j,-j,k,-k\}| =6$,  $|E|=12$, $\mathrm{deg}(i)=\mathrm{deg}(j)=\mathrm{deg}(k)=4$  and we confirm \cite[Propositions 2.3, 2.6]{aam} noting that $\Gamma_G$ is planar and regular. In order to compute $\nu_2$, fix $\xi=i$ and $x=j$. Here  $\rho_i(y) = \rho_i(j)=1$ for all $y \in V$ so that $\mu^i_j=0$. But when $\xi=i$ and $x=-i$, $1=\rho_i(y) < \rho_i(-i)=2$ for all $y \in V-\{-i\}$ and so $\mu^i_{-i}=1$. Since this argument may be repeated for $x=-j$ and $x=-k$, we conclude that $\nu_2=4$. This means that $c=\frac{4^{3/4}}{4^8 \ e^8}$. Here $\Omega=V$ confirms Theorem \ref{spec1} by 
$12 \ge  \left(\frac{4^{3/4}}{4^8 \ e^8}\right)  \cdot {24}^{3/4} .$
\end{ex}

The following is the first example of  Sobolev inequality for $\Gamma_G$.

\begin{cor} $\Gamma_G$ satisfies the thesis of Theorem \ref{sob1} with $R_0=\iota=\omega'=\sigma_{xy}=1$, $\mu_x=\mathrm{deg}(x)$, $v_0=1+\sum_{x \in \Omega} \mathrm{deg}(x)$, $n=\nu_2$.
\end{cor}

One of the most interesting problems is due to the optimality of the constants which appear in Theorem \ref{sob1}. This hasn't been discussed properly in \cite{chung4}, but the same authors have produced a series of papers in the last ten years on the problem of weakening  $P(\iota, \delta, R_0)$. Recently, some new  metric spaces are considered in \cite{ambrosio1, ambrosio2} and they seem to be the natural contexts where the above property can be generalized. We don't discuss this delicate aspect here.

\section{Second result}

The reader may observe that the condition $P(\delta, \iota, R_0)$ implies the isoperimetry, as explained in \cite[Theorem 6.3]{chung4}, but, on the other hand, $(\flat)$ has the form of a Sobolev--Poincar\'e inequality when  $K(\Omega_0)=0$ (see \cite{hebey} for details). This motivates us to characterize  a special situation, by means of another well known inequality of Nash type. The following theorem illustrates such equivalence.

\begin{thm}\label{t2}If a weighted graph $(\Gamma, \sigma)$ has  $P(\delta, \iota, R_0)$ with $\Omega \subseteq V$ of $\mu(\Omega) < \infty$ and $p=2n/n-2$, then the following conditions are equivalent for any $f \in \mathbb{R}^V$
\[ \left(\sum_{x \in V}|f(x)|^p \mu_x \right)^{\frac{2}{p}} \le A(p) \ {\underset{x,y \in V}{\underset{y : y \sim x} \sum}} | f(y)-f(x)|^2 \sigma_{xy}, \leqno{(\dag)}\]
\[ \left(\sum_{x \in V}|f(x)|^2 \mu_x \right)^{1+\frac{2}{n}} \le B(p) \ \left({\underset{x,y \in V}{\underset{y : y \sim x} \sum}} |f(y)-f(x)|^2 \sigma_{xy}\right) \ \left( \sum_{x \in V} |f(x)| \mu_x \right)^{\frac{4}{n}}, \leqno{(\dag\dag)}\]
where $A(p)$ and $B(p)$ are  (nonoptimal) constants depending only on $p$.
\end{thm}

\begin{proof}The property $P(\delta, \iota, R_0)$ is assumed, in order to be sure that there exists a graph satisfying an isoperimetric inequality (see Theorem \ref{iso1}), and, so, by Theorem \ref{sob1}, a Sobolev type inequality. In fact the proof of the equivalence among the conditions $(\dag)$ and $(\dag\dag)$, as we will see, doesn't use the property $P(\delta, \iota, R_0)$. On the other hand, we put it in the assumptions of the theorem for this precise motivation.

\medskip
\medskip 

$(\dag) \Rightarrow (\dag \dag)$.  We apply the H\"older inequality in the following form:
\[\sum_{x \in V}|f(x)|^2 \mu_x=\sum_{x \in V}|f(x)|^{\frac{p}{p-1}+\frac{p-2}{p-1}} \mu_x\]
\[ \le \left(\sum_{x \in V}\left(|f(x)|^{\frac{p}{p-1}}\right)^{p-1} \mu_x \right)^{\frac{1}{p-1}} \ \left( \sum_{x \in V} \left(|f(x)|^{\frac{p-2}{p-1}}\right)^{\frac{p-1}{p-2}}\mu_x \right)^{\frac{p-2}{p-1}},\]
that is,  
\[\sum_{x \in V}|f(x)|^2 \mu_x \le \left(\sum_{x \in V}|f(x)|^p \mu_x \right)^{\frac{1}{p-1}} \ \left( \sum_{x \in V} |f(x)| \mu_x \right)^{\frac{p-2}{p-1}}\]
and by $(\dag)$ we upper bound the right side of the above inequality with
\[\le \left(\left( A(p) \ {\underset{x,y \in V}{\underset{y : y \sim x} \sum}} | f(y)-f(x)|^2 \sigma_{xy}\right)^\frac{p}{2}\right)^{\frac{1}{p-1}} \left( \sum_{x \in V} |f(x)| \mu_x \right)^{\frac{p-2}{p-1}} \]
so that the $\left( 1+\frac{2}{n}\right)$th power implies
\[\left(\sum_{x \in V}|f(x)|^2 \mu_x \right)^{ 1+\frac{2}{n}}\]
\[ \le \left( A(p) \ {\underset{x,y \in V}{\underset{y : y \sim x} \sum}} | f(y)-f(x)|^2 \sigma_{xy}\right)^{^{\left( 1+\frac{2}{n}\right) \left(\frac{1}{p-1}\right) \left( \frac{p}{2}\right) }} \left( \sum_{x \in V} |f(x)| \mu_x \right)^{^{\left( 1+\frac{2}{n}\right)  \left(\frac{p-2}{p-1}\right)}} \]
\[\le  A(p) \ \left({\underset{x,y \in V}{\underset{y : y \sim x} \sum}} |f(y)-f(x)|^2 \sigma_{xy}\right) \ \left( \sum_{x \in V} |f(x)| \mu_x \right)^{\frac{4}{n}}, \]
since \[\left( 1+\frac{2}{n}\right) \left(\frac{1}{p-1}\right)\left( \frac{p}{2}\right)= \left( 1+\frac{2}{n}\right) \left(\frac{1}{\frac{2n}{n-2}-1}\right) \left(\frac{n}{n-2}\right)\]\[=\left(\frac{n+2}{n} \right)\left(\frac{n-2}{n+2} \right) \left(\frac{n}{n-2}\right)=1\]
and
\[\left( 1+\frac{2}{n}\right)  \left(\frac{p-2}{p-1}\right) = \left(\frac{n+2}{n} \right) \left(\frac{\frac{2n}{n-2}-2}{\frac{2n}{n-2}-1}\right) = \left(\frac{n+2}{n} \right)\left(\frac{\frac{4}{n-2}}{\frac{n+2}{n-2}}\right)=\frac{4}{n}.\] Therefore $(\dag \dag)$ follows with $A(p)=B(p)$.

\medskip
\medskip

$(\dag \dag) \Rightarrow (\dag)$. Given $f \in \mathbb{R}^V$ and $k \in \mathbb{Z}$, we define
$U_k=\{x \in V \ | \ |f(x)|<2^k \}$, $V_k=\{x \in V \ | \ 2^k \le |f(x)|<2^{k+1} \}$, $W_k=\{x \in V \ | \  |f(x)|\ge 2^{k+1} \}$ and
\[f_k(x)= \begin{cases}  0,  & \ \ \ \mathrm{if} \ \ x \in U_k\\
|f_k(x)|-2^k, & \ \ \ \mathrm{if} \ \ x \in V_k \\
2^k,  & \  \ \ \mathrm{if} \ \ x \in W_k.
\end{cases}\]
We note some useful properties of the way of writing $f_k(x)$ as above. Firstly,  $V=U_k \dot{\cup} V_k \dot{\cup} W_k$, that is, $V$ is the disjoint union of the sets $U_k$, $V_k$ and $W_k$. Secondly, $f_k(x)$ is zero over  $U_k$, and this doesn't give contribution in writing sums, while $f_k(x)$ is constant over $W_k$ once $k$ is fixed. Thirdly, we have by construction that $W_{k+1} \subseteq W_k$ for all $k \in \mathbb{Z}$. From the first of these properties, we get easily that 
\[{\underset{x,y \in V}{\underset{y : y \sim x} \sum}} |f_k(y)-f_k(x)|^2 \sigma_{xy} 
= {\underset{x,y \in U_k}{\underset{y : y \sim x} \sum}} |f_k(y)-f_k(x)|^2 \sigma_{xy}  + {\underset{x,y \in V_k}{\underset{y : y \sim x} \sum}} |f_k(y)-f_k(x)|^2 \sigma_{xy} \]
\[ + {\underset{x,y \in W_k}{\underset{y : y \sim x} \sum}} |f_k(y)-f_k(x)|^2 \sigma_{xy}  \le {\underset{x,y \in V_k}{\underset{y : y \sim x} \sum}} |f_k(y)-f_k(x)|^2 \sigma_{xy} \]
Now  we apply $(\dag \dag)$  to each $f_k(x)$ and, because of the above inequality, we find
\[ \left(\sum_{x \in V}|f_k(x)|^2 \mu_x \right)^{1+\frac{2}{n}} \le B(p)  \left( {\underset{x,y \in V_k}{\underset{y : y \sim x} \sum}} |f_k(y)-f_k(x)|^2 \sigma_{xy}\right)  \left( \sum_{x \in V} |f_k(x)| \mu_x \right)^{\frac{4}{n}}. \leqno{(*)}\]
Estimating the $\left(1+\frac{2}{n}\right)$th rooth of the term on the left side of $(*)$, we get 
\[  2^{2k} \ \underbrace{\sum_{x \in V} \chi_{_{W_k}}(x) \mu_x}_{\mu(W_k)}  \le  \sum_{x \in W_k}|f_k(x)|^2 \mu_x \le \sum_{x \in V}|f_k(x)|^2 \mu_x \]
and so we have the following lower bound for this term:
\[  \left(2^{2k} \ \sum_{x \in V} \chi_{_{W_k}}(x) \mu_x\right)^{1+\frac{2}{n}}  \le  \left( \sum_{x \in V}|f_k(x)|^2 \mu_x\right)^{1+\frac{2}{n}}.  \]
On the other hand, we  estimate the ($4/n$)th rooth of the following term in the right side of $(*)$:
\[\sum_{x \in V} |f_k(x)| \mu_x =\sum_{x \in U_k} |f_k(x)| \mu_x + \sum_{x \in V_k} |f_k(x)| \mu_x + \sum_{x \in W_k} |f_k(x)| \mu_x\]
\[= \sum_{x \in V_k} |f_k(x)| \mu_x + \sum_{x \in W_k} |f_k(x)| \mu_x   \le 2^k \sum_{x \in V} \chi_{_{V_k}}(x) \mu_x + 2^k \sum_{x \in V} \chi_{_{W_k}}(x) \mu_x \]
\[\le 2^k \sum_{x \in V} \chi_{_{W_{k-1}}}(x) \mu_x\]
and so we have the following upper bound for this term:
\[\left(\sum_{x \in V} |f_k(x)| \mu_x\right)^{\frac{4}{n}} \le \left(2^k \sum_{x \in V} \chi_{_{W_{k-1}}}(x) \mu_x\right)^{\frac{4}{n}}.\]
We conclude that $(*)$  implies
\[\left(2^{2k} \ \sum_{x \in V} \chi_{_{W_k}}(x) \mu_x\right)^{1+\frac{2}{n}}  \le B(p)  \left( {\underset{x,y \in V_k}{\underset{y : y \sim x} \sum}} |f_k(y)-f_k(x)|^2 \sigma_{xy}\right)   \left(2^k \sum_{x \in V} \chi_{_{W_{k-1}}}(x) \mu_x\right)^{\frac{4}{n}}. \leqno{(**)}\]
In order to manipulate the terms which appear in $(**)$, we denote
\[a_k= 2^{pk} \sum_{x \in V} \chi_{_{W_{k-1}}}(x) \mu_x \ \  \mathrm{and} \  \ b_k= {\underset{x,y \in V_k}{\underset{y : y \sim x} \sum}} |f_k(y)-f_k(x)|^2 \sigma_{xy}, \]
where $p$ is always equal to $2n/(n-2)$. Now
\[a_{k+1}= 2^{pk+p} \sum_{x \in V} \chi_{_{W_k}}(x) \mu_x\]
and we rewrite $(**)$ as
\[2^{(2k-pk-p) \ \left(1+\frac{2}{n}\right) } \ \ (a_{k+1})^{1+\frac{2}{n}} \le   
 B(p)  \left( {\underset{x,y \in V_k}{\underset{y : y \sim x} \sum}} |f_k(y)-f_k(x)|^2 \sigma_{xy}\right)   \ 2^{\frac{4k}{n}-\frac{4pk}{n}} (a_k)^{\frac{4}{n}},\]
that is,
\[ (a_{k+1})^{1+\frac{2}{n}} \le   
\ 2^{\frac{4k}{n}-\frac{4pk}{n}+(-2k+pk+p) \ \left(1+\frac{2}{n}\right) }  \  B(p)  \left( {\underset{x,y \in V_k}{\underset{y : y \sim x} \sum}} |f_k(y)-f_k(x)|^2 \sigma_{xy}\right)   \ (a_k)^{\frac{4}{n}}.\]
Since $\left(1+\frac{2}{n}\right) \left(\frac{n}{n+2} \right)=1$, we do the $\left(\frac{n}{n+2} \right)$th power and get
\[ a_{k+1} \le   
\ 2^{\frac{4k}{n+2} (1-p)-2k+pk+p  }  \  B(p)^{\frac{n}{n+2}} \ \left( {\underset{x,y \in V_k}{\underset{y : y \sim x} \sum}} |f_k(y)-f_k(x)|^2 \sigma_{xy}\right)^{\frac{n}{n+2} }   \ (a_k)^{ \frac{4}{n+2} },\]
but 
\[\frac{4k}{n+2}(1-p) -2k+pk+p=\frac{4k}{\frac{2p}{p-2}+2}(1-p) -2k+pk+p\]
\[ \frac{4k}{\frac{4p-4}{p-2}}(1-p) -2k+pk+p =\left(\frac{p-2}{p-1}\right) (1-p) k -2k+pk+p=-k(p-2)-2k+pk+p=p  \]
and so 
\[ a_{k+1} \le   
\ 2^p  \  B(p)^{\frac{n}{n+2}} \ \left( {\underset{x,y \in V_k}{\underset{y : y \sim x} \sum}} |f_k(y)-f_k(x)|^2 \sigma_{xy}\right)^{\frac{n}{n+2} }   \ (a_k)^{ \frac{4}{n+2} }. \leqno{(\sharp)}\]
Until now, we have shown that $(\sharp)$ follows from $(*)$ via $(**)$. But we may sum $(\sharp)$ over $k \in \mathbb{Z}$ and get
\[\sum_{k \in  \mathbb{Z}} a_k=\sum_{k \in  \mathbb{Z}} a_{k+1} \le   
\ 2^p  \  B(p)^{\frac{n}{n+2}} \ \sum_{k \in  \mathbb{Z}} (b_k)^{\frac{n}{n+2}}   (a^2_k)^{\frac{2}{n+2}} \]
and applying the H\"older inequality with conjugate exponents $P=\frac{n}{n+2}$ and $Q=1-\frac{n}{n+2}=\frac{2}{n+2}$,  this quantity is upper bounded by
\[ \le    2^p  \  B(p)^{\frac{n}{n+2}} \ \left(\sum_{k \in  \mathbb{Z}} b_k  \right)^{\frac{n}{n+2}}   \left( \sum_{k \in  \mathbb{Z}} a^2_k  \right)^{\frac{2}{n+2}}\]
where we may even upper bound the last term  a priori, getting 
\[ \le    2^p  \  B(p)^{\frac{n}{n+2}} \ \left(\sum_{k \in  \mathbb{Z}} b_k  \right)^{\frac{n}{n+2}}   \left( \sum_{k \in  \mathbb{Z}} a_k  \right)^{\frac{4}{n+2}}.\]
This allows us to conclude
\[\sum_{k \in  \mathbb{Z}} a_k \le    2^p  \  B(p)^{\frac{n}{n+2}} \ \left(\sum_{k \in  \mathbb{Z}} b_k  \right)^{\frac{n}{n+2}}   \left( \sum_{k \in  \mathbb{Z}} a_k  \right)^{\frac{4}{n+2}} \]
hence
\[\sum_{k \in  \mathbb{Z}} a_k \le  \left(2^p \  B(p)^{\frac{n}{n+2}} \   \left(\sum_{k \in  \mathbb{Z}} b_k  \right)^{\frac{n}{n+2}}\right)^{\frac{n+2}{n-2}}= 2^\frac{p(n+2)}{(n-2)} \  B(p)^{\frac{n}{n-2}} \   \left(\sum_{k \in  \mathbb{Z}} b_k  \right)^{\frac{n}{n-2}}. \leqno{(\sharp \sharp)} \]
Now, on a hand $\bigcup_{k \in \mathbb{Z}}V_k=V$ and so
\[\sum_{k \in  \mathbb{Z}} b_k= \sum_{k \in \mathbb{Z}} \left({\underset{x,y \in V_k}{\underset{y : y \sim x} \sum}} |f_k(y)-f_k(x)|^2 \sigma_{xy}\right)  \le {\underset{x,y \in V}{\underset{y : y \sim x} \sum}} |f(y)-f(x)|^2 \sigma_{xy}, \]
  on another hand, we note that $\bigcup_{k \in \mathbb{Z}}V_k=V$, that $V_k=W_{k-1} - W_k$ and that the restriction $|f|^p(V_k) \le (2^{k+1})^p=2^p(2^{kp})$, and so
 \[  \sum_{x \in V} |f(x)|^p \mu_x =  \sum_{k \in \mathbb{Z}}\left( \sum_{x \in V_k} |f(x)|^p \mu_x\right)  \le 
\sum_{k \in \mathbb{Z}}2^p  (2^{kp})  \underbrace{\left(\sum_{x \in V} \chi_{_{V_k}}(x)\mu_x\right)}_{\mu(V_k)}\]
\[=2^p \ \sum_{k \in \mathbb{Z}}  2^{kp}  \underbrace{\left(\sum_{x \in V} \chi_{_{W_{k-1}-W_k}}(x)\mu_x\right)}_{\mu(W_{k-1}-W_k)}= 2^p \sum_{k \in \mathbb{Z}} \left(a_k  -\frac{2^{(k+1)p}}{2^p} \underbrace{\left(\sum_{x \in V} \chi_{_{W_k}}(x)\mu_x\right)}_{\mu(W_k)} \right) \]
\[=2^p \sum_{k \in \mathbb{Z}} \left( a_k - \frac{a_{k+1}}{2^p}\right)=2^p\left(1-\frac{1}{2^p}\right) \sum_{k \in \mathbb{Z}}a_k=(2^p-1)\sum_{k \in \mathbb{Z}}a_k.\]
Therefore we combine these last two inequalities with $(\sharp \sharp)$ and find that
\[\left( \sum_{x \in V} |f(x)|^p \mu_x\right)^{\frac{2}{p}} \le (2^p-1)^{\frac{2}{p}} \ \ 2^{2(p-1)} \ B(p) \ {\underset{x,y \in V}{\underset{y : y \sim x} \sum}} |f(y)-f(x)|^2 \sigma_{xy}, \]
which gives exactly $(\dag)$ when $A(p)=(2^p-1)^{\frac{2}{p}} \ \ 2^{2(p-1)} \ B(p) $. 
\end{proof}

The following corollary shows a Nash inequality for $\Gamma_G$ for the first time.

\begin{cor} $\Gamma_G$ satisfies the thesis of Theorem \ref{t2} with $n=\nu_2$, $\mu_x=\mathrm{deg}(x)$, $\sigma_{xy}=\iota=R_0=\delta=1$.
\end{cor}

\begin{proof}Application of definitions, Lemma \ref{l:1} and Theorem \ref{t2}.
\end{proof}

\section*{Acknowledgements}
The second author thanks CAPES for the project 061/2013 and the institutes of mathematics of UFRJ and IMPA  in Rio de Janeiro (Brasil) for their hospitality. In particular, we are very grateful to Prof. Walcy Santos, who allowed us to collaborate intensively on the subject of the present paper.

\end{document}